\documentclass[a4paper,12pt]{article}
\usepackage{mathrsfs}
\usepackage{amssymb}
\usepackage{amsthm}
\usepackage{amsmath}
\usepackage{amsfonts}
\usepackage{latexsym}
\usepackage{amscd}
\usepackage{fancyhdr}
 \newtheorem{thm}{Theorem}
 \newtheorem{cor}[thm]{Corollary}
 \newtheorem{lem}[thm]{Lemma}

\begin{document}

\title{\textbf{ \large LEFT APP-RINGS OF SKEW GENERALIZED POWER SERIES}
\thanks{Partially supported by National Natural Science Foundation of China(10961021). }}
\author{\small\textbf{Renyu Zhao  } \\
{\small College of Economics and Management,}\\ \small Northwest
Normal
University, Lanzhou 730070, P.R. China \\
\small E-mail: renyuzhao026@gmail.com\\}
\date{}
\maketitle

\def\abstractname{Abstract}
\begin{abstract} \noindent A ring $R$ is called a left APP-ring if the left annihilator
$l_{R}(Ra)$ is right $s$-unital as an ideal of $R$ for any $a\in R$.
Let $R$ be a ring, $(S,\leq)$ a strictly ordered monoid and
$\omega:S\longrightarrow {\rm End}(R)$ a monoid homomorphism. The
skew generalized power series ring $[[R^{S,\leq},\omega]]$ is a
common generalization of (skew) polynomial rings, (skew) power
series rings, (skew) Laurent polynomial rings, (skew) group rings,
and Malcev-Neumann Laurent series rings. We study the left
APP-property of the skew generalized power series ring
$[[R^{S,\leq},\omega]]$. It is shown that if $(S,\leq)$ is a
strictly totally ordered monoid, $\omega:S\longrightarrow {\rm
Aut}(R)$ a monoid homomorphism and $R$ a ring satisfying descending
chain condition on right annihilators, then $[[R^{S,\leq},\omega]]$
is left APP if and only if for any $S$-indexed subset $A$ of $R$,
the ideal $l_{R}\big(\sum_{a\in A}\sum_{s\in S}R\omega_{s}(a)\big)$
is right $s$-unital.

\vspace*{5mm}

\noindent \emph{Key Words:} left APP-ring, skew generalized power
series ring.
 \vspace*{3mm}

\noindent\emph{Mathematics Subject Classification:} 16W60

\end{abstract}

\vspace*{5mm}

\setcounter{section}{1}
\subsection*{1. Introduction and preliminaries}
\indent

Throughout this paper, $R$ denotes a ring (not necessarily
commutative) with unity. For a nonempty subset $X$ of $R$, $l_R(X)$
and $r_R(X)$ denote the left and right annihilator of $X$ in $R$,
respectively. We will denote by ${\rm End}(R)$ the monoid of ring
endomorphisms of $R$, and by ${\rm Aut}(R)$ the group of ring
automorphisms of $R$.

Recall that a ring $R$ is a right (resp. left) PP-ring if the right
(resp. left) annihilator of an element of $R$ is generated by an
idempotent. The ring $R$ is called a PP-ring if it is both right and
left PP. A ring $R$ is called (quasi-) Baer if the left annihilator
of every nonempty subset (every left ideal) of $R$ is generated by
an idempotent of $R$. For more details and examples of PP-rings,
Baer rings and quasi-Baer rings, see
\cite{arm74,bir-kim-park-cm-00,bir-kim-park-jpaa-01,bir-par-ja-03,fraser-nich-mj-89,hong-kim-kwak-jpaa-00}.
As a generalization of quasi-Baer rings, G.F. Birkenmeier, J.Y. Kim
and J.K. Park in \cite{bir-kim-park-comm-01} introduced the concept
of left principally quasi-Baer rings. A ring $R$ is called left
principally quasi-Baer (or simply, left p.q.-Baer) if the left
annihilator of a principal left ideal of $R$ is generated by an
idempotent. Similarly, right p.q.-Baer rings can be defined. A ring
$R$ is called p.q.-Baer if it is both right and left p.q.-Baer.
Observe that biregular rings and quasi-Baer rings are left
p.q.-Baer. According to \cite{ezeh-ijms-88}, a commutative ring $R$
is called PF-ring if the annihilator $ann_R(a)$ is pure as an ideal
of $R$ for every $a\in R$. As a common generalization of left
p.q.-Baer rings, right PP-rings and PF-rings, the concept of left
APP-rings was introduced in \cite{liu-zhao-glasgow-06}. A ring $R$
is called a left APP-ring if the left annihilator $l_R(Ra)$ is pure
as a left ideal of $R$ for every $a\in R$. For more details and
examples of left APP-rings, see \cite{liu-zhao-glasgow-06} and
\cite{hir-jpaa-02}.

There are a lot of results concerning left principal quasi-Baerness,
right PP-property and PF-property of polynomial extensions and power
series extensions of a ring. We recall some of them as follows. It
was proved in \cite[Theorem 2.1]{bir-kim-park-kmj-00} that a ring
$R$ is left p.q.-Baer if and only if $R[x]$ is left p.q.-Baer. C.Y.
Hong, N.K. Kim and T.K. Kwak showed in \cite[Corollary
15]{hong-kim-kwak-jpaa-00} that if $\sigma$ is a rigid endomorphism
of $R$, then $R$ is a left p.q.-Baer ring if and only if $R[x;
\sigma, \delta]$ is a left p.q.-Baer ring. G.F. Birkenmeier and J.K.
Park in \cite[Theorem 1.2]{bir-par-ja-03} showed that if $M$ is a
u.p.-monoid, then $R[M]$ is left p.q.-Baer if and only if $R$ is
left p.q.-Baer. For skew monoid rings it was proved in \cite[Theorem
5]{liu-yang-glasgow-10} that if $M$ is an ordered monoid and
$\phi:M\longrightarrow {\rm Aut}(R)$ a monoid homomorphism, then the
skew monoid ring $R*M$ induced by the monoid homomorphism $\phi$ is
a left p.q.-Baer ring if and only if $l_{R}\big(\sum_{g\in
M}R\phi(g)(a)\big)$ is generated by an idempotent of $R$ for every
$a\in R$. If all right semicentral idempotents of $R$ are central,
then it was proved in \cite[Theorem 3]{liu-comm-02} that $R[[x]]$ is
left p.q.-Baer if and only if $R$ is left p.q.-Baer and every
countable family of idempotents of $R$ has a generalized join in the
set of all idempotents of $R$. It was proved in \cite[Theorem
A]{arm74} that $R$ is a reduced PP-ring if and only if $R[x]$ is a
reduced PP-ring. J. Fraser and W.K. Nicholson in
\cite{fraser-nich-mj-89} showed that $R[[x]]$ is a reduced PP-ring
if and only if $R$ is a reduced PP-ring and every countable family
of idempotents of $R$ has a least upper bound in the set of all
idempotents of $R$. It was proved in \cite{ezeh-ijms-88} that
$R[[x]]$ is a PF-ring if and only if for any two countable subsets
$A$ and $B$ of $R$ with $A\subseteq ann_{R}(B)$, there exists $r\in
ann_{R}(B)$ such that $ar=a$ for all $a\in A$.

In recent years, many researches have carried out an extensive study
of rings of (skew) generalized power series (for example, P.
Ribenboim\cite{rib-ja-92,rib-ja-95,rib-ja-97}, Z.K.
Liu\cite{liu-ahsan-ams-00,liu-glasgow-02,liu-annals-02,liu-ams-06},
H. Kim\cite{kim-jpaa-02,kim-kwon-kyungpook-07}, R. Mazurek and M.
Ziembowski\cite{marks-mazurk-ziembowski-bams-10,mazurk-ziembowski-ja-06,mazurk-ziembowski-ja-07,mazurk-ziembowski-comm-08},
and the present author \cite{zhao-liu-tjm-08,zhao-jiao-tjm}, etc.).
In particular, it was shown in \cite[Corollary 3.8]{zhao-jiao-tjm}
that if $(S,\leq)$ is a strictly totally ordered monoid and $R$ a
ring satisfying the condition that $ab=0\Longleftrightarrow
a\omega_{s}(b)=0$ for any $a,b\in R$ and any $s\in S$, then
$[[R^{S,\leq},\omega]]$ is a left p.q.-Baer ring if and only if for
any $S$-indexed set $A$ of $R$, $l_{R}(\sum_{a\in A}Ra)$ is
generated by an idempotent of $R$. In \cite[Corollary
5.5]{zhao-liu-tjm-08}, we proved that $[[R^{S,\leq}]]$ is a reduced
PP-ring if and only if $R$ is a reduced PP-ring and for every
$S$-indexed subset $A$ of idempotents of $R$, $ann_{R}(A)$ is
generated by an idempotent of $R$ if and only if $R$ is a reduced
PP-ring and any $S$-indexed subset of idempotents of $R$ has a least
upper bound in the set of all idempotents of $R$. H. Kim and T.I.
Kwon proved in \cite[Theorem 2.4]{kim-kwon-kyungpook-07} that if
$(S,\leq)$ is a strictly totally ordered monoid, then
$[[R^{S,\leq}]]$ is a PF-ring if and only if for any two $S$-indexed
subsets $A$ and $B$ of $R$ with $A\subseteq ann_{R}(B)$, there
exists $r\in ann_{R}(B)$ such that $ar=a$ for all $a\in A$.

For left APP-rings, it was proved in \cite[Theorem
2]{liu-yang-glasgow-10} that if $M$ is an ordered monoid and
$\phi:M\longrightarrow {\rm Aut}(R)$ is a monoid homomorphism, then
the skew monoid ring $R*M$ is a left APP-ring if and only if for any
$b\in R$, $l_{R}\big(\sum_{g\in M}R\phi(g)(b)\big)$ is pure as a
left ideal of $R$. It was noted in \cite[Example
2.4]{liu-zhao-glasgow-06} that there exists a commutative von
Neumann regular ring $R$ (hence left APP), but the ring $R[[x]]$ is
not APP. In \cite[Theorem 2]{liu-yang-am-06}, it was shown that if
$R$ is a ring satisfying descending chain condition on right
annihilators then $R[[x,\alpha]]$ is a left APP-ring if and only if
for any sequence $(b_{0},b_{1},\ldots)$ of elements of $R$ the ideal
$l_{R}\big(\sum_{j=0}^{\infty}\sum_{k=0}^{\infty}R\alpha^{k}(b_{j})\big)$
is pure as a left ideal of $R$, where $\alpha\in {\rm Aut}(R)$.

In this note, we will consider left APP-property of skew generalized
power series rings. We will show that if $(S,\leq)$ is a strictly
totally ordered monoid, $\omega:S\longrightarrow {\rm Aut}(R)$ a
monoid homomorphism and $R$ is a ring satisfying descending chain
condition on right annihilators, then $[[R^{S,\leq},\omega]]$ is
left APP if and only for any $S$-indexed subset $A$ of $R$, the
ideal $l_{R}\big(\sum_{a\in A}\sum_{s\in S}R\omega_{s}(a)\big)$ is
pure as a left ideal of $R$.

In order to recall the skew generalized power series ring
construction, we need some definitions. Let $(S,\leq)$ be a
partially ordered set. Recalled that $(S,\leq)$ is artinian if every
strictly decreasing sequence of elements of $S$ is finite, and that
$(S,\leq)$ is narrow if every subset of pairwise order-incomparable
elements of $S$ is finite. Let $S$ be a commutative monoid. Unless
stated otherwise, the operation of $S$ shall be denoted additively,
and the neutral element by 0. The following definition is due to
\cite{rib-ja-97},\cite{liu-ams-06} and
\cite{mazurk-ziembowski-comm-08}.

Let $R$ be a ring, $(S,\leq)$ a strictly ordered monoid (that is,
$(S,\leq)$ is an ordered monoid such that if $s,s',t\in S$ and
$s<s'$, then $s+t<s'+t$), and $\omega:S\longrightarrow {\rm End}(R)$
a monoid homomorphism. For any $s\in S$, let $\omega_s$ denote the
image of $s$ under $\omega$, that is $\omega_s=\omega(s)$. Consider
the set $A$ of all maps $f:S \longrightarrow R$ whose support
$\mathrm{supp}(f)=\{s\in S\mid f(s)\neq 0\}$ is artinian and narrow.
Then for any $s\in S$ and $f,g\in A$ the set $$X_{s}(f,g)=\{(u,v)\in
S\times S\mid u+v=s,f(u)\neq 0,g(v)\neq 0\}$$ is finite. This fact
allows to define the operation of convolution as
follows:$$(fg)(s)=\sum\limits_{(u,v)\in X_{s}(f,g)}
f(u)\omega_u\big(g(v)\big),\qquad {\rm if} \ \
X_{s}(f,g)\neq\emptyset$$ and $(fg)(s)=0$ if $X_{s}(f,g)=\emptyset$.
With this operation and pointwise addition, $A$ becomes a ring,
which is called the ring of skew generalized power series with
coefficients in $R$ and exponents in $S$, and we denote by
$[[R^{S,\leq},\omega]]$.

The skew generalized power series construction embraces a wide range
of classical ring-theoretic extensions, including skew polynomial
rings, skew power series rings, skew Laurent polynomial rings, skew
group rings, Malcev-Neumann Laurent series rings and of course the
"untwisted" versions of all of these.

If $(S,\leq)$ is a strictly totally ordered monoid and $0\neq
f\in[[R^{S,\leq},\omega]]$, then supp$(f)$ is a nonempty
well-ordered subset of $S$. We denote $\pi(f)$ the smallest element
of supp$(f)$. To any $r\in R$ and any $s\in S$ we associated the
maps $\lambda_{r}^{s}\in [[R^{S,\leq},\omega]]$ defined by
$$\lambda_{r}^{s}(t)=\bigg\{\begin{array}{cc}
  r, & t=s, \\
  0, & t\neq s, \\
\end{array}\quad  t\in S.$$ In particular, denote $c_{r}=\lambda_{r}^{0},\;e_{s}=\lambda_{1}^{s}$. It is clear
that $r\longmapsto c_{r}$ is a ring embedding of $R$ into
$[[R^{S,\leq},\omega]]$, $s\longmapsto e_{s}$ is a monoid embedding
of $S$ into the multiplicative monoid of ring
$[[R^{S,\leq},\omega]]$, and $\lambda_{r}^{s}=c_re_s$,
$e_sc_r=c_{\omega_{s}(r)}e_s$.

\setcounter{section}{1}
\subsection*{2. Main Results}
\indent

 An ideal $I$ of $R$ is said to be  right $s$-unital if,
for each $a\in I$ there exists an element $x\in I$ such that $ax=a$.
Note that if $I$ and $J$ are right $s$-unital ideals, then so is
$I\cap J$ (if $a\in I\cap J$, then $a\in aIJ\subseteq a(I\cap J)$).
It follows from \cite[Theorem 1]{tom76} that $I$ is right $s$-unital
if and only if for any finitely many elements $a_1, a_2, \ldots,
a_n\in I$ there exists an element $x\in I$ such that $a_i=a_ix, i=1,
2, \ldots, n$. A submodule $N$ of a left $R$-module $M$ is called a
 pure submodule if $L\otimes_RN\longrightarrow L\otimes_RM$ is
a monomorphism for every right $R$-module $L$. By \cite[Proposition
11.3.13]{ste75}, an ideal $I$ is right $s$-unital if and only if
$R/I$ is flat as a left $R$-module if and only if $I$ is pure as a
left ideal of $R$.

By \cite{liu-zhao-glasgow-06}, a ring $R$ is called a left APP-ring
if the left annihilator $l_R(Ra)$ is right $s$-unital as an ideal of
$R$ for any element $a\in R$.

Right APP-rings may be defined analogously. Clearly every left
p.q.-Baer ring is a left APP-ring (thus the class of left APP-rings
includes all biregular rings and all quasi-Baer rings). If $R$ is a
commutative ring, then $R$ is APP if and only $R$ is FP. From
\cite[Proposition 2.3]{liu-zhao-glasgow-06} it follows that right
PP-rings are left APP and left APP-rings are quasi-Armendariz in the
sense that whenever $f(x)=a_0+a_1x+\cdots +a_mx^m$,
$g(x)=b_0+b_1x+\cdots +b_nx^n\in R[x]$ satisfy $f(x)R[x]g(x)=0$, we
have  $a_iRb_j=0$ for each $i$ and $ j$ (see, for example
\cite{hir-jpaa-02}). For more details on left APP-rings, see
\cite{liu-zhao-glasgow-06,hir-jpaa-02}.

\begin{lem}
Let $(S,\leq)$ be a strictly totally ordered monoid and
$\omega:S\longrightarrow{\rm Aut}(R)$ a monoid homomorphism. If
$l_R\big(\sum_{s\in S}R\omega_s(a)\big)$ is right $s$-unital for any
$a\in R$, then for any $f,g\in [[R^{S,\leq},\omega]]$ satisfy
$g[[R^{S,\leq},\omega]]f=0$, $g(u)\omega_u(R\omega_s(f(v)))=0$ for
any $u,v,s\in S$.\end{lem}

\begin{proof} Let $0\neq f,g\in [[R^{S,\leq},\omega]]$ be
such that $g[[R^{S,\leq},\omega]] f=0$. Assume that $\pi(g)=u_{0}$
and $\pi(f)=v_{0}$. Then for any $(u,v)\in X_{u_{0}+v_{0}}(g,f)$,
$u_{0}\leq u,\ v_{0}\leq v.$ If $u_{0}<u$, since $\leq$ is a strict
order, $u_{0}+v_{0}<u+v_{0}\leq u+v=u_{0}+v_{0},$ a contradiction.
Thus $u=u_{0}$. Similarly, $v=v_{0}.$ Hence, for any $r\in R$ and
any $s\in S$,
$$\begin{aligned}0=(g \lambda_{r}^{s} f)(u_{0}+s+v_{0})&=\sum_{(u,v)\in
X_{u_{0}+s+v_{0}}(g,\lambda_{r}^{s}f)}g(u)\omega_u\big(r\omega_s(f(v))\big)\\
&=g(u_{0})\omega_{u_{0}}\big(r\omega_s(f(v_{0}))\big).\end{aligned}$$

Now let $w\in S$ with $u_{0}+v_{0}\leq w$. Assume that for any $u\in
\text{supp}(g)$ and any $v\in \text{supp}(f)$, if $u+v<w$, then
$g(u)\omega_u(R\omega_s(f(v)))=0$ for any $s\in S$. We will show
that $g(u)\omega_u(R\omega_s(f(v)))=0$ for any $s\in S$, any $u\in
\text{supp}(\phi)$ and any $ v\in \text{supp}(f)$ with $u+v=w$. For
convenience, we write
$$X_{w}(g,f)=\{(u_{i},v_{i})\mid
i=1,2,\ldots,n\}$$ with $v_{1}<v_{2}<\cdots<v_{n}$ (Note that if
$v_{1}=v_{2}$, then from $u_{1}+v_{1}=u_{2}+v_{2}$ it follows that
$u_{1}=u_{2}$, and thus $(u_{1},v_{1})=(u_{2},v_{2})$). Then for any
$r\in R$ and any $s\in S$,
$$\begin{aligned}0=(g \lambda_{r}^sf)(s+w)&=\sum_{(u,v)\in
X_{w}(g,\lambda_{r}^{s}f)}g(u)\omega_u\big(r\omega_{s}(f(v))\big)\\
&=\sum_{i=1}^{n}g(u_{i})\omega_{u_{i}}\big(r\omega_{s}(f(v_{i}))\big).
\end{aligned}\eqno(1)$$ Note that $u_i+v_1<u_i+v_i=w$ for each $i=2,\ldots,n$. Then by
induction hypothesis, $g(u_i)\omega_{u_i}(R\omega_t(f(v_1)))=0$ for
any $t\in S$ and each $i=2,\ldots,n$. Thus
$\omega_{u_i}^{-1}(g(u_i))\in l_R\big(\sum_{t\in
S}R\omega_t(f(v_1))\big)$ since $\omega_{u_i}\in {\rm Aut}(R)$ for
any $i=2,\ldots,n.$ Hence there exists $e_{1}\in l_R\big(\sum_{t\in
S}R\omega_t(f(v_1))\big)$ such that $g(u_i)=g(u_i)\omega_{u_i}(e_1)$
for $i=2,\ldots,n$ by the hypothesis. Let $r'\in R$, take
$r=e_{1}r'$ in the equation (1), we have
$$0=\sum_{i=1}^{n}g(u_{i})\omega_{u_{i}}\big(e_{1}r'\omega_{s}(f(v_{i}))\big)
=\sum_{i=2}^{n}g(u_{i})\omega_{u_{i}}\big(r'\omega_{s}(f(v_{i}))\big).\eqno(2)$$
Since $u_{i}+v_{2}<u_{i}+v_{i}=w$ for any $i=3,\ldots,n$, by
hypothesis, there exists $e_{2}\in l_R\big(\sum_{t\in
S}R\omega_t(f(v_2))\big)$ such that $g(u_i)=g(u_i)\omega_{u_i}(e_2)$
for $i=3,\ldots,n$. Hence take $r'=e_{2}r''$ in (2) where $r''\in
R$, we deduced that
$$\sum_{i=3}^{n}g(u_{i})\omega_{u_{i}}\big(r''\omega_{s}(f(v_{i}))\big)=0.$$
Continuing in this manner yields that
$g(u_{n})\omega_{u_{n}}\big(R\omega_{s}(f(v_{n}))\big)=0$ for any
$s\in S$. Consequently, for any $s\in S$,
$$g(u_{n-1})\omega_{u_{n-1}}\big(R\omega_{s}(f(v_{n-1}))\big)=0,\ldots,g(u_{1})\omega_{u_{1}}\big(R\omega_{s}(f(v_{1}))\big)=0.$$

Therefore, by transfinite induction, we have shown that
$g(u)\omega_{u}\big(R\omega_{s}(f(v))\big)=0$ for any $u,v,s\in S$.
\end{proof}

\begin{lem} Let $(S,\leq)$ be a strictly ordered monoid and $\omega:S\longrightarrow{\rm End}(R)$ a monoid
homomorphism. If $[[R^{S,\leq},\omega]]$ is a left APP-ring and $S$
is cancellative, then $l_R\big(\sum_{s\in S}R\omega_s(a)\big)$ is
right $s$-unital for any $a\in R$.
\end{lem}

\begin{proof} Let $a\in R$ and $b\in l_R\big(\sum_{s\in S}R\omega_s(a)\big)$. Then $c_b[[R^{S,\leq},\omega]]c_a=0$.
Since $[[R^{S,\leq},\omega]]$ is left APP, there exists an $h\in
l_{[[R^{S,\leq},\omega]]}\big([[R^{S,\leq},\omega]]c_a\big)$ such
that $c_b=c_bh$. Then $b=c_b(0)=(c_bh)(0)=bh(0)$ and, for any $r\in
R$, any $s\in S$,
$$0=(h\lambda_{r}^{s}c_a)(s)=h(0)r\omega_s(a),$$ which imply that
$l_R\big(\sum_{s\in S}R\omega_s(a)\big)$ is right $s$-unital for any
$a\in R$.
\end{proof}

Let $(S,\leq)$ be a strictly ordered monoid and $A$ a nonempty
subset of $R$. We will say $A$ is $S$-indexed, if there exists an
artinian and narrow subset $I$ of $S$ such that $A$ is indexed by
$I$.

\begin{thm} Let $(S,\leq)$ be a strictly totally ordered monoid and
$\omega:S\longrightarrow{\rm Aut}(R)$ a monoid homomorphism. If $R$
satisfies descending chain condition on right annihilators, then the
following conditions are equivalent:
\\\indent {\rm (1)} $[[R^{S,\leq},\omega]]$ is a left APP-ring.
\\\indent {\rm (2)} For any $S$-indexed subset $A$ of $R$, $l_R\big(\sum_{a\in A}\sum_{s\in S}R\omega_s(a)\big)$ is
right $s$-unital.
 \end{thm}
\begin{proof} (2)$\Longrightarrow$ (1). Assume that
$f,g\in[[R^{S,\leq},\omega]]$ are such that
$g[[R^{S,\leq},\omega]]f=0$. Then, by the hypothesis and Lemma 1,
$g(u)\omega_u(R\omega_s(f(v)))=0$ for any $u,v,s\in S$. Since
$\omega_u\in {\rm Aut}(R)$, $\omega_u^{-1}(g(u))R\omega_s(f(v))=0$
for any $u,v,s\in S$. Thus for any $u\in {\rm supp}(g)$,
$$\omega_u^{-1}(g(u))\in l_R\Big(\sum_{v\in {\rm supp}(f)}\sum_{s\in
S}R\omega_s(f(v))\Big).$$ Let $$\mathscr{D}=\big\{r_R(Y)|Y\subseteq
\{\omega_u^{-1}(g(u))|u\in {\rm supp}(g)\},\;|Y|<\infty\big\}.$$
Then $\mathscr{D}$ is a nonempty set of right annihilators. Since
$R$ satisfies descending chain condition on right annihilators,
$\mathscr{D}$ has a minimal element, say $r_R(Y_0)$. Assume that
$Y_0=\{\omega_{u_1}^{-1}(g(u_1)),\omega_{u_2}^{-1}(g(u_2)),\ldots,\omega_{u_n}^{-1}(g(u_n))\}$.
Then $$\omega_{u_i}^{-1}(g(u_i))\in l_R\Big(\sum_{v\in {\rm
supp}(f)}\sum_{s\in S}R\omega_s(f(v))\Big),\;\;i=1,2,\ldots,n.$$
Thus, by (2), there exists $e\in l_R\Big(\sum_{v\in {\rm
supp}(f)}\sum_{s\in S}R\omega_s(f(v))\Big)$ such that
$$\omega_{u_i}^{-1}(g(u_i))=\omega_{u_i}^{-1}(g(u_i))e,\;\;
i=1,2,\ldots,n.$$ If ${\rm supp}(g)=\{u_1,u_2,\ldots,u_n\}$, then
for all $u\in{\rm supp}(g),
\omega_{u}^{-1}(g(u))=\omega_{u}^{-1}(g(u))e$. Now assume that
$u\in{\rm supp}(g)\backslash\{u_1,u_2,\ldots,u_n\}$. Then, by the
minimality of $r_R(Y_0)$,
$$r_R\big(\omega_{u_1}^{-1}(g(u_1)),\ldots,\omega_{u_n}^{-1}(g(u_n)),\omega_{u}^{-1}(g(u))\big)=
r_R\big(\omega_{u_1}^{-1}(g(u_1)),\ldots,\omega_{u_n}^{-1}(g(u_n))\big).$$
Thus $\omega_{u}^{-1}(g(u))=\omega_{u}^{-1}(g(u))e$. This implies
that $ \omega_{u}^{-1}(g(u))=\omega_{u}^{-1}(g(u))e$ for any
$u\in{\rm supp}(g)$. Thus for any $h\in[[R^{S,\leq},\omega]]$ and
any $t\in S$,
$$(c_ehf)(t)=\sum_{(s,v)\in X_{t}(h,f)}eh(s)\omega_{s}(f(v))=0,$$
and
$$(gc_e)(t)=g(t)\omega_{t}(e)=\omega_{t}\big(\omega_{t}^{-1}(g(t))e\big)=\omega_{t}\big(\omega_{t}^{-1}(g(t))\big)=g(t),$$
which imply that $c_e\in
l_{[[R^{S,\leq},\omega]]}\big([[R^{S,\leq},\omega]]f\big)$ and
$g=gc_e$. Hence $[[R^{S,\leq},\omega]]$ is a left APP-ring.

(1)$\Longrightarrow$ (2). Let $A=\{a_t| t\in I\}$ be an $S$-indexed
subset of $R$. Define $f\in [[R^{S,\leq},\omega]]$ via
$$f(t)=\bigg\{\begin{array}{cc}
  a_{t},              & t\in I ,\\
  0,                  & t\notin I. \\
\end{array}$$
Let $b\in l_R\big(\sum_{t\in I}\sum_{s\in S}R\omega_s(a_t)\big)$.
Then $c_b[[R^{S,\leq},\omega]]f=0$. Since $[[R^{S,\leq},\omega]]$ is
left APP, there exists an $h\in
l_{[[R^{S,\leq},\omega]]}\big([[R^{S,\leq},\omega]]f\big)$ such that
$c_b=c_bh$. Thus $b=c_b(0)=(c_bh)(0)=bh(0)$. By (1), Lemma 2 and
Lemma 1, $h(u)\omega_u\big(R\omega_s(f(t))\big)=0$ for any $u,s,t\in
S$. In particular, $h(0)R\omega_s(f(t))=0$ for any $s,t\in S$. This
implies that $h(0)\in l_R\big(\sum_{t\in I}\sum_{s\in
S}R\omega_s(f(t))\big)$. Thus (2) holds.
\end{proof}

\begin{cor} {\rm(\cite[Theorem 2 ]{liu-yang-am-06})} Let $R$ be a ring
satisfying descending chain condition on right annihilators and
$\alpha\in {\rm Aut}(R)$. Then the following conditions are
equivalent:
\\\indent {\rm (1)} $R[[x;\alpha]]$ is a left APP-ring.
\\\indent {\rm (2)} For any countable subset $A$ of $R$, $l_R\big(\sum_{a\in A}\sum_{i=0}^{\infty}R\alpha^i(a)\big)$ is
right $s$-unital.\end{cor}

\begin{cor} Let $R$ be a ring satisfying descending chain condition on right
annihilators and $\alpha\in {\rm Aut}(R)$. Then the following
conditions are equivalent:
\\\indent {\rm (1)} $R[[x,x^{-1};\alpha]]$ is a left APP-ring.
\\\indent {\rm (2)} For any countable subset $A$ of $R$, $l_R\big(\sum_{a\in A}\sum_{i=-\infty}^{\infty}R\alpha^i(a)\big)$ is
right $s$-unital.\end{cor}

Let $\alpha$ and $\beta$ be ring automorphisms of $R$ such that
$\alpha\beta=\beta\alpha$. Let $S =(\mathbb{N}\cup\{0\})\times
(\mathbb{N}\cup\{0\})$ (resp. $\mathbb{Z}\times\mathbb{Z}$) be
endowed the lexicographic order, or the reverse lexicographic order,
or the product order of the usual order of $\mathbb{N}\cup\{0\}$
(resp. $\mathbb{Z}$), and define $\omega: S \longrightarrow {\rm
Aut}(R)$ via $\omega(m, n)=\alpha^{m}\beta^{n}$ for any
$m,n\in\mathbb{N}\cup\{0\}$ (resp. $m, n \in \mathbb{Z}$). Then
$[[R^{S,\leq},\omega]]=R[[x,y;\alpha,\beta]]$ (resp. $R[[x, y,
x^{-1}, y^{-1};\alpha,\beta]]$), in which
$(ax^{m}y^{n})(bx^{p}y^{q}) = a\alpha^{m}\beta^{n}(b)x^{m+p}y^{n+q}$
for any $m, n, p, q\in \mathbb{N}\cup\{0\}$ (resp. $m, n, p, q\in
\mathbb{Z}$).

\begin{cor} Let $R$ be a ring satisfying descending chain condition on right
annihilators, $\alpha$ and $\beta$ be ring automorphisms of R such
that $\alpha\beta=\beta\alpha$. Then the following conditions are
equivalent:
\\\indent {\rm (1)} $R[[x,y;\alpha,\beta]]$ (resp.
$R[[x, y, x^{-1}, y^{-1};\alpha,\beta]]$) is a left APP-ring.
\\\indent {\rm (2)} For any countable subset $A$ of $R$,
$l_R\big(\sum_{a\in
A}\sum_{i=0}^{\infty}\sum_{j=0}^{\infty}R\alpha^i\beta^j(a)\big)$
(resp. $l_R\big($ $\sum_{a\in
A}\sum_{i=-\infty}^{\infty}\sum_{j=-\infty}^{\infty}R\alpha^i\beta^j(a)\big)$)
is right $s$-unital.\end{cor}

If $S$ the multiplicative monoid $(\mathbb{N},\cdot)$, endowed with
the usual order $\leq$, then $[[R^{(\mathbb{N},\cdot),\leq}]]$ is
the ring of arithmetical functions with values in $R$, endowed with
the Dirichlet convolution:$$(fg)(n)=\sum_{d\mid
n}f(d)g(n/d),\quad\text{for each } n\geq 1.$$

\begin{cor} Let $R$ be a ring satisfying descending chain condition on right
annihilators. Then the following conditions are equivalent:
\\\indent {\rm (1)} $[[R^{(\mathbb{N},\cdot),\leq}]]$ is a left APP-ring.
\\\indent {\rm (2)} For any countable subset $A$ of $R$,
$l_R\big(\sum_{a\in A}Ra\big)$ is right $s$-unital.\end{cor}

 Let $(S,\leq)$ be a strictly totally ordered monoid
which is also artinian. Then the set $X_s = \{(u, v)|u+v = s, u, v
\in S\}$ is finite for any $s\in S$. Let $V$ be a free Abelian
additive group with the base consisting of elements of $S$. It was
noted in \cite{liu-comm-00} that V is a coalgebra over $\mathbb{Z}$
with the comultiplication map and the counit map as follows:
$$\triangle(s)=\sum_{(u,v)\in X_{s}}u\otimes v,\qquad \epsilon(s)=\bigg\{\begin{array}{cc}
  1,              & s=0 ,\\
  0,                  & s\neq 0, \\
\end{array}$$and $[[R^{S,\leq}]] \cong {\rm Hom}(V,R)$, the dual algebra with multiplication
$$f\ast g=(f\otimes g)\triangle \;\;\;\; \forall f,g\in {\rm Hom}(V,R).$$

\begin{cor} Let $(S,\leq)$ be a strictly totally ordered monoid
which is also artinian, $R$ a ring satisfying descending chain
condition on right annihilators  and ${\rm Hom}(V,R)$ defined as
above. Then the following conditions are equivalent:
\\\indent {\rm (1)} ${\rm Hom}(V,R)$ is a left APP-ring.
\\\indent {\rm (2)} For any $S$-indexed subset $A$ of $R$,
$l_R\big(\sum_{a\in A}Ra\big)$ is right $s$-unital.\end{cor}

\vskip 3mm \centerline{ {\bf ACKNOWLEDGMENT} } \vskip 3mm The author
wishes to express his sincere thanks to the referee for his/her
valuable suggestions.

\end{document}